\documentclass[10pt]{article}

\usepackage[english]{babel}
\usepackage[utf8]{inputenc}
\usepackage[T1]{fontenc}

\usepackage{amsmath, amsthm, amssymb, mathtools}

\usepackage[shortlabels]{enumitem}

\usepackage{cases}

\newcommand\EC{\textbf{(EC)}\xspace}
\newcommand\NC{\textbf{(NC)}\xspace}

\usepackage[
backend=biber,
firstinits=true,
style=alphabetic,
sorting=nyt,
isbn=false,
url=false,
doi=true,
maxbibnames=99
]{biblatex}
\AtBeginDocument{\def\MR#1{}}

\usepackage{hyperref}

\usepackage{sectsty}

\usepackage{tikz}
\usetikzlibrary{positioning}
\usepackage{caption}

\usepackage[a4paper]{geometry}

\providecommand{\keywords}[1]{\textbf{Keywords.} #1}
\providecommand{\MSC}[1]{\textbf{2020 Mathematics Subject Classification.} #1}

\newtheorem{theorem}{Theorem}[section]

\newtheorem{proposition}[theorem]{Proposition}
\newtheorem{example}[theorem]{Example}

\theoremstyle{definition}
\newtheorem{definition}[theorem]{Definition}
\newtheorem{remark}[theorem]{Remark}

\renewcommand\epsilon{\varepsilon}
\renewcommand\mapsto{\longmapsto}

\usepackage{color}

\usepackage{xspace}

\newcommand{\R}{\field{R}\xspace}

\newcommand{\field}[1]{\ensuremath{\mathbb{#1}}}

\newcommand{\ens}[1]{ \left\{#1\right\} }

\newcommand\diag{\mathrm{diag} \,}

\newcommand{\Tinf}{T_{\mathrm{inf}}}

\newcommand{\TI}{T_{\mathrm{inf}} ^{\mathrm{I}}}
\newcommand{\TII}{T_{\mathrm{inf}}^{\mathrm{II}}}

\newcommand{\yIm}{y^{\mathrm{I}}_-}
\newcommand{\yIIm}{y^{\mathrm{II}}_-}
\newcommand{\yIp}{y^{\mathrm{I}}_+}
\newcommand{\yIIp}{y^{\mathrm{II}}_+}

\newcommand{\uI}{u^{\mathrm{I}}}
\newcommand{\uII}{u^{\mathrm{II}}}

\newcommand{\LambdaIm}{\Lambda^{\mathrm{I}}_-}
\newcommand{\LambdaIIm}{\Lambda^{\mathrm{II}}_-}\newcommand{\LambdaIp}{\Lambda^{\mathrm{I}}_+}
\newcommand{\LambdaIIp}{\Lambda^{\mathrm{II}}_+}

\newcommand\TCN{T_{\scriptscriptstyle \mathrm{[CN]} }}

\newcommand\pt[1]{\frac{\partial #1}{\partial t}}
\newcommand\px[1]{\frac{\partial #1}{\partial x}}

\newcommand\pxi[1]{\frac{\partial #1}{\partial \xi}}

\newcommand\Tau{\mathcal{T}}

\newcommand{\rank}{\mathrm{rank} \,}

\newcommand\Id{\mathrm{Id}}

\newcommand\abs[1]{\left|#1\right|}
\newcommand\st{\quad \middle| \quad}

\newcommand\setCN{\mathcal{B}}

\newcommand\dds{\frac{d}{ds}}

\newcommand\Xm[2]{\ell_{#1}\left[#2\right]}

\newcommand\Qcan{Q^c}
\newcommand\Gcan{G^c}
\newcommand\gcan{g^c}

\newcommand\IC{\textbf{+IC}}

\title{Minimal control time of some 1D first-order linear hyperbolic systems with row rank deficient boundary coupling matrix}

\author{
Long Hu\thanks{School of Mathematics, Shandong University, Jinan, Shandong 250100, China.
E-mail: \texttt{hul@sdu.edu.cn}}
\and
Guillaume Olive\thanks{Faculty of Mathematics and Computer Science, Jagiellonian University, ul. {\L}ojasiewicza 6, 30-348 Krak\'{o}w, Poland.
E-mail: \texttt{math.golive@gmail.com} or \texttt{guillaume.olive@uj.edu.pl}}
\and
Zengzhi Xin\thanks{Mathematical Research Center, Shandong University, Jinan, Shandong 250100, China.
E-mail: \texttt{xinzengzhi@mail.sdu.edu.cn}}
}

\date{July 20, 2026}

\begin{document}

\maketitle

\begin{abstract}
The minimal control time of one-dimensional first-order linear hyperbolic systems is by now well-known when the boundary coupling matrix is full row rank.
The goal of this article is to characterize the minimal control time for a class of systems with row rank deficient boundary coupling matrices.
\end{abstract}


\vspace{0.2cm}
\MSC{35L40; 93B05}

\tableofcontents

\section{Introduction and main results}

\subsection{Problem description}

\paragraph{General framework.}

The system we are interested in is the following system of one-dimensional (1D) hyperbolic Partial Differential Equations (PDEs):
\begin{equation}\label{syst}
\begin{dcases}
\pt{y}(t,x)+\Lambda(x) \px{y}(t,x)=M(x) y(t,x), \\
y_-(t,1)=u(t), \quad y_+(t,0)=Qy_-(t,0), \\
y(0,x)=y^0(x),
\end{dcases}
\quad (t,x) \in R_T.
\end{equation}
The notations are as follows.
The domain is the rectangle $R_T=(0,T) \times (0,1)$, where $T>0$.
$t$ is the time variable and $x$ is the space variable.
Since we are interested in systems, the state $y:R_T \to \R^n$ is thus a vector-valued function ($n \geq 2$).

The $n \times n$ matrix $\Lambda$ will always be assumed diagonal $\Lambda =\diag(\lambda_1,\ldots,\lambda_n)$, with $m \geq 1$ negative and $p \geq 1$ positive entries ($m+p=n$), also called speeds, and satisfying
\begin{equation}\label{hyp speeds}
\lambda_1(x)<\ldots<\lambda_m(x)<0<\lambda_{m+1}(x)<\ldots<\lambda_{m+p}(x),
\end{equation}
for every $x$.
The speeds will be assumed to be Lipschitz continuous on $\R$, so that the characteristics are globally defined: for every $1 \leq k \leq n$ and $(t,x) \in \R^2$, there is a unique $C^1$ function $s \mapsto \chi_k(s;t,x)$ solution to the Ordinary Differential Equation (ODE)
$$
\begin{dcases}
\dds \chi_k(s;t,x)=\lambda_k(\chi_k(s;t,x)), \quad \forall s \in \R, \\
\chi_k(t;t,x)=x.
\end{dcases}
$$

The $n \times n$ matrix $M$ couples the equations of the system inside the domain and will be called the internal coupling matrix.
We will assume that $M \in L^{\infty}(0,1)^{n \times n}$.

The structure of $\Lambda$ induces a natural splitting of the state $y$ into components corresponding to negative and positive speeds:
$$
y=
\begin{pmatrix}
y_- \\
y_+ 
\end{pmatrix},
\quad
\begin{array}{l}
y_- \in \R^m,
\\
y_+ \in \R^p.
\end{array}
$$
This explains why, for system \eqref{syst} to be well-posed, we need to consider boundary conditions at $x=1$ for $y_-$ and at $x=0$ for $y_+$.
The particular type of (local) boundary conditions considered in \eqref{syst} is also motivated by applications (\cite{BC16,YK22}).
The function $u:(0,T) \to \R^m$ is called the control, it is a function that will be at our disposal.
Since it acts on the boundary of the domain, we speak of boundary controllability.
We have $m$ controls for $n$ equations so that we have less controls than equation.
Such systems are called underactuated.
The system is coupled on the other part of the boundary by the matrix $Q \in \R^{p \times m}$.
This matrix will be called the boundary coupling matrix.

We shall also refer to systems of the form \eqref{syst} as $m+p$ systems.

System \eqref{syst} is well-posed in several functional settings.
In this paper, we will work in $L^2$ for the state and the controls.
We can prove that, for every $y^0 \in L^2(0,1)^n$ and $u \in L^2(0,T)^m$, there exists a unique solution
$$
y \in C^0([0,T];L^2(0,1)^n) \cap C^0([0,1];L^2(0,T)^n).
$$
By solution we mean ``solution along the characteristics'' (sometimes also called ``broad solution'', see \cite[Chapter 3]{Bre00}).
We refer for instance to \cite[Appendix A]{CHOS21} for a proof of the above well-posedness result in such a setting (see also \cite[Appendix A]{BC16} when $u=0$ and \cite[Lemma 3.2]{CN19} in the $L^\infty$ setting).

\paragraph{Control problems.}

Since $y(t,\cdot) \in L^2(0,1)^n$ for all $t$, we can consider control problems in the state space $L^2(0,1)^n$:

\begin{definition}
Let $T>0$ be fixed.
We say that system \eqref{syst} is null controllable (below, \NC) in time $T$ if, for every $y^0 \in L^2(0,1)^n$, there exists $u \in L^2(0,T)^m$ such that the corresponding solution $y$ to system \eqref{syst} satisfies the final condition $y(T,\cdot)=0$.
\end{definition}

We also say that system \eqref{syst} is exactly controllable (below, \EC) in time $T$ if, for every $y^1, y^0 \in L^2(0,1)^n$, there exists $u \in L^2(0,T)^m$ such that the corresponding solution $y$ to system \eqref{syst} satisfies the final condition $y(T,\cdot)=y^1$.
Obviously, \EC always implies \NC.
In general, the converse implication does not hold.
Nevertheless, for our system \eqref{syst}, we have equivalence if $\rank Q=p$ (see e.g. \cite[Remark 1.3]{HO21-JMPA}).

It is not difficult to check that controllability in time $T_1$ implies controllability in any time $T_2 \geq T_1$.
It is therefore natural and fundamental to try to find the smallest possible control time, the so-called minimal control time.

\begin{definition}
We denote by $\Tinf \in [0,+\infty]$ the minimal time for the \NC of system \eqref{syst}:
$$
\Tinf=
\inf\ens{T>0 \st \text{System \eqref{syst} is \NC in time $T$}}.
$$
\end{definition}

We say ``minimal'' control time according to the current literature, despite this time is not always a minimal element of the set.
We keep this naming here, but we use the notation with the ``inf'' to avoid eventual confusions.

It is not difficult to see that, whether we consider \EC or \NC, the minimal control time for a single equation (the transport equation) with speed $\lambda_k$ is the solution $T_k$ to $\chi_k(T_k;0,1)=0$ if $\lambda_k<0$ and to $\chi_k(T_k;0,0)=1$ if $\lambda_k>0$.
In any case, it is explicitly given by
$$
T_k=\int_0^1 \frac{1}{\abs{\lambda_k(\xi)}} \, d\xi.
$$
We refer for instance to \cite[Section 3]{Rus78}, \cite[Section 3.1]{CVKB13} or \cite[Chapter 1]{Oli25-notes}.
These quantities will play a key role below.
Note that assumption \eqref{hyp speeds} implies the following order relation:
$$
T_1<\cdots<T_m \quad \text{ and } \quad
T_{m+1}>\cdots>T_{m+p}.
$$

\subsection{Literature}

Controllability of 1D first-order hyperbolic systems has been widely studied in the literature.
It dates back at least to the article \cite{Rus67} and the celebrated survey \cite{Rus78}.
Therein, the author showed that system \eqref{syst} is \NC in any time $T \geq T_{m+1}+T_m$.

However, as observed by the author himself, this control time is in general larger than the minimal control time.
This is already easy to see in the trivial case $M=Q=0$, for which $\max\ens{T_m,T_{m+1}}$ is clearly the minimal null control time (all the equations are decoupled).
The first interesting situation where the minimal control time was found is the case $M=0$ in \cite{Wec82}.
This problem was then revisited in \cite{HO22-JDE}.
Therein, the authors introduced the important notion of canonical form of boundary coupling matrices.
A boundary coupling matrix $Q \in \R^{p \times m}$ is in canonical form if it has at most one nonzero entry on each row and each column, and this entry is equal to $1$.
It can be shown that, for every $Q \in \R^{p \times m}$, there exists a unique $\Qcan \in \R^{p \times m}$ which is in canonical form and such that $LQU=\Qcan$ for some unit lower triangular matrix $L \in \R^{p \times p}$ and some invertible upper triangular matrix $U \in \R^{m \times m}$.
The existence part is just a consequence of the Gaussian elimination, the uniqueness part is less straightforward and a proof can be found in \cite[Section 2.6]{Oli25-notes}.
$\Qcan$ is called the canonical form of $Q$.
The authors in \cite{HO22-JDE} then showed that
$$M=0 \quad \Longrightarrow \quad \Tinf=\max_{1 \leq k \leq \rho} \ens{T_m, \quad T_{m+1}, \quad T_{m+r_k}+T_{c_k}},$$
where $(r_1,c_1), \ldots, (r_{\rho},c_{\rho})$ denote the positions of the nonzero entries of $\Qcan$, with $r_1<\cdots<r_{\rho}$, and $\rho=\rank Q$.

The minimal control time has now been found in several other situations.
A rather detailed presentation can be found in the recent survey \cite{HO26-CAMB-pre}.
Two results are of particular interest for the present article:

\begin{itemize}
\item
In \cite{CN19}, the authors introduced the class $\setCN$ of boundary coupling matrices $Q \in \R^{p \times m}$ that satisfy: the submatrix formed from the first $k$ rows and columns of $Q$ is invertible, for all $k \leq \min\ens{p,m-1}$.
When $m \leq p$ and $Q \in \setCN$, they proved in \cite[Theorem 1]{CN21} that system \eqref{syst} is \NC in any time $T>\TCN$, where
$$\TCN=\max_{1 \leq k \leq m} T_{m+k}+T_k.$$

\item
When there exists $\rho_0$ such that the first $\rho_0$ rows of $Q$ are linearly independent and $c_{\rho_0}=m$, it follows from \cite[Corollary 1.12]{HO22-JDE} that the minimal control time for the \NC of system \eqref{syst} does not depend on $M$ and is thus given by
$$
\Tinf=\max_{1 \leq k \leq \rho_0} T_{m+k}+T_{c_k}.
$$

\end{itemize}

The goal of this article is to generalize the above result of \cite{CN21} in various aspects.
First, we will just assume that $Q_1,\ldots,Q_{m-1}$ are linearly independent ($m \geq 2$).
Second, we will actually find the exact value of the minimal control time.
Before going on, we also need to be sure that our framework is not already covered by the above result of \cite{HO22-JDE}.
Thus, we have to assume in addition that $c_k<m$ for all $k \leq m-1$.
Equivalently, this means that
\begin{equation}\label{main hyp}
\text{The submatrix formed from the first $m-1$ rows and columns of $Q$ is invertible.}
\end{equation}
This will be our framework.
For instance, $Q$ can be one of the following matrices:
$$
\begin{pmatrix}
1 & 0 \\
0 & 0
\end{pmatrix}, \quad
\begin{pmatrix}
1 & 0 \\
0 & 0 \\
0 & 0
\end{pmatrix}, \quad
\begin{pmatrix}
1 & 0 \\
0 & 0 \\
0 & 1
\end{pmatrix}.
$$

\subsection{Main result and comments}

Our main result is the following:
\begin{theorem}\label{main thm}
Assume that $2 \leq m \leq p$ and $Q$ satisfies \eqref{main hyp}.
Then, the minimal time for the \NC of system \eqref{syst} is
$$\Tinf=\max\ens{\tau^1, \tau^2},$$
where
$$
\tau^1=\max_{1 \leq k \leq m-1} T_{m+k}+T_{c_k},
\quad
\tau^2=\max\ens{\max_{\substack{m \leq k \leq p \\ k<r_m}}  T_{m+k}^{I_{m+k}}+T_m, \quad T_{m+r_m}+T_m},
$$
for some intervals $I_{m+k} \subset (0,1)$ to be defined below, and where
$T_{m+k}^{I_{m+k}}=\int_{I_{m+k}} \frac{1}{\lambda_{m+k}(\xi)} \, d\xi$.

\end{theorem}

In the definition of $\tau^2$, we use the convention that the undefined quantities are simply not taken into account. More precisely,
\begin{itemize}
\item
If $r_m$ does not exist (i.e. $\rank Q=m-1$), then $\tau^2=\max_{m \leq k \leq p}  T_{m+k}^{I_{m+k}}+T_m$.

\item
If $r_m=m$, then $\tau^2=T_{2m}+T_m$.
\end{itemize}
We shall also use these conventions below without explicitly referring to it for the sake of the presentation.

\begin{remark}
Any $Q \in \setCN$ satisfies the assumptions of the theorem.
Therefore, it generalizes \cite[Theorem 1]{CN21} recalled above in the state of the art.
In fact, this quantifies exactly the gap between $\TCN$ and $\Tinf$.
\end{remark}

Theorem \ref{main thm} gives a complete characterization of the \NC of system \eqref{syst} and a precise definition of the intervals $I_{m+k}$ will be given during the proof below.
In general, it can still be challenging to characterize these intervals explicitly in function of the parameters $\Lambda,M,Q$.
Let us present an example where everything can be explicitly computed:

\begin{example}\label{ex}
Let $m=p=2$ and $f \in L^{\infty}(0,1)$ be arbitrary.
Consider
$$
Q=
\begin{pmatrix}
1 & 0 \\
0 & 0 
\end{pmatrix},
\quad
M=
\left(\begin{array}{cc|cc}
0 & 0 & \times & \times \\
\times & 0 & \times & \times \\
\hline
\times & 0 & 0 & \times \\
\times & f & \times & 0 
\end{array}\right).
$$
Then, from Theorem \ref{main thm}, we have 
$$
\Tinf=\max\ens{T_3+T_1, \, T_4^{I_4}+T_2}.
$$
Moreover, we shall see in Section \ref{sect example} below that, thanks to the special structure of this system, the interval $I_4$ verifies the identity
$$T_4^{I_4}+T_2=T_4^{J_4}+T_2^{J_4},
\quad J_4=(\Xm{x_0}{f},1),$$
where $x_0 \in (0,1)$ is the unique solution to $T_4^{(0,x_0)}+T_2^{(0,x_0)}=T_4$ and $\Xm{x_0}{f}$ is the length of the largest interval of the form $(0, \times) \subset (0,x_0)$ where the function $f$ vanishes a.e.
\end{example}

\paragraph{Organization of the paper.}
The rest of this article is organized as follows.
Section \ref{sect prelim} is a preparatory section where we recall some results from \cite{HVDMK19}, \cite{HO21-JMPA}, \cite{HO21-COCV} and \cite{HO22-JDE}.
In Section \ref{sect: BC for K} we prove a new and important result on special solutions to the so-called kernel equations.
Our main result is then proved in Section \ref{sect proof main thm}.
Finally, the above example is detailed in Section \ref{sect example}.

\paragraph{Notations.}
All along this paper, for a matrix (or matrix-valued function) denoted by an uppercase letter $A$, we denote its $(i,j)$-th entry by the corresponding lowercase letter $a_{ij}$.

For brevity, we will often not explicitly write the initial condition at $t=0$ of our systems below and just write ``\IC'' instead.

\section{Background results}
\label{sect prelim}

In this part, we recall previous results of the literature that will be needed to understand the proof of our main result.

\subsection{Equivalent system by the backstepping method}

First of all, as in many other works about the controllability or stabilization of hyperbolic system \eqref{syst}, the starting point is the backstepping method for PDEs.
Popularized with the reference book \cite{KS08} (see also the survey \cite{VABAK26}) and then developed specifically for hyperbolic systems in \cite{CVKB13,DMVK13,HDM15,HDMVK16,HVDMK19}, this method consists in using a Volterra integral transformation of the second kind to obtain an equivalent system for which the controllability properties are simpler to discuss.
Notably, it was established in \cite{HVDMK19} that the controllability in time $T$ of system \eqref{syst} is equivalent to the controllability in the same time of the following systems:
\begin{equation}\label{syst G}
\begin{dcases}
\pt{y}(t,x)+\Lambda(x) \px{y}(t,x)=G(x) y_-(t,0), \\
y_-(t,1)=u(t), \quad y_+(t,0)=Qy_-(t,0), \\
\IC,
\end{dcases}
\quad (t,x) \in R_T,
\end{equation}
where:
\begin{itemize}
\item
$G:(0,1) \to \R^{n \times m}$ is given by
\begin{equation}\label{def G}
G(x)
=
K(x,0)P, \quad P=-\Lambda(0)\begin{pmatrix} \Id_m \\ Q \end{pmatrix}.
\end{equation}

\item
$K:\Tau \to \R^{n \times n}$, where $\Tau=\ens{(x,\xi) \st 0<\xi<x<1}$, is the kernel of the Volterra integral transformation mapping system \eqref{syst} to system \eqref{syst G} and it has to be a solution to the so-called kernel equations:
\begin{equation}\label{kern equ}
\begin{dcases}
\Lambda(x)\px{K}(x,\xi)
+\pxi{K}(x,\xi)\Lambda(\xi)
+K(x,\xi) (\Lambda'(\xi)+M^0(\xi))=0,
\\
\Lambda(x) K(x,x)-K(x,x)\Lambda(x) =M^0(x),
\end{dcases}
\quad (x,\xi) \in \Tau.
\end{equation}
We will discuss the existence of such a solution shortly.

\item
The new internal coupling matrix $M^0:(0,1) \to \R^{n \times n}$ appearing in the equations above is just a slight modification of $M$ to remove its diagonal entries.
It is necessary for the existence of a solution to the kernel equations, as is immediately apparent from the boundary condition at $(x,x)$.
It is explicitly given by
$$
M^0(x)=(D(x)M(x) +\Lambda(x)D'(x)) D(x)^{-1},
$$
where $D=\diag(d_1,\ldots,d_n)$ with $d_k(x)=e^{-\int_0^x \frac{m_{kk}(\xi)}{\lambda_k(\xi)} \,d\xi}$.
\end{itemize}

The advantage of system \eqref{syst G} over system \eqref{syst} is its simpler internal coupling structure, achieved through the term $G(x) y_-(t,0)$ instead of $M(x) y(t,x)$.
We refer for instance to \cite[Section 5.1]{HO26-CAMB-pre} for a more detailed discussion.

As result, from now on, we will only focus on the controllability properties of systems of the form \eqref{syst G}.
We shall however come back to the precise link between $G$ and $\Lambda,M,Q$ to treat the example in Section \ref{sect example} below.

Let us now comment the existence of a solution to the kernel equations \eqref{kern equ}, which is the main difficulty in the backstepping method.
First observe that the rows of $K$ in system \eqref{kern equ} are decoupled.
Thus, it is enough to fix a row $i$ and consider the corresponding equations for $\begin{pmatrix} k_{i1} & \ldots & k_{in} \end{pmatrix}$.
From \cite[Section VI]{HDMVK16} (see also \cite[Remark A.2]{HVDMK19} for space-varying speeds), we can extract the following result:

\begin{theorem}\label{thm kern}
Fix a row index $i \in \ens{1,\ldots,n}$.
Let $a_{\alpha j}, f_{ij}, \varphi_{ij}$, for $\alpha, j \in \ens{1,\ldots,n}$, be arbitrary functions of $L^{\infty}(0,1)$.
Then, there exists a unique solution $\begin{pmatrix} k_{i1} & \ldots & k_{in} \end{pmatrix} \in L^{\infty}(\Tau)^{1 \times n}$ to
$$
\begin{dcases}
\lambda_i(x)\px{k_{ij}}(x,\xi)
+\pxi{k_{ij}}(x,\xi)\lambda_j(\xi)
+\sum_{\alpha=1}^n k_{i\alpha}(x,\xi) a_{\alpha j}(\xi)=0,
\\
k_{ij}(x,x) =f_{ij}(x),\quad \text{ if } j \neq i,
\end{dcases}
\quad (x,\xi) \in \Tau,
$$
with the following additional boundary conditions:
\begin{itemize}
\item
For $i \leq m$:
$$
k_{ij}(1,\xi) =\varphi_{ij}(\xi) \quad \text{ if } j<i,
\quad
k_{ij}(x,0) =\varphi_{ij}(x) \quad  \text{ if } i \leq j \leq m.
$$

\item
For $i \geq m+1$:
$$
k_{ij}(x,0) =\varphi_{ij}(x) \quad \text{ if } m+1 \leq j \leq i, \quad
k_{ij}(1,\xi) =\varphi_{ij}(\xi) \quad  \text{ if } j>i.
$$
\end{itemize}

\end{theorem} 

\begin{remark}\label{rmk wp kern}
For the present work we will actually have to consider in Section \ref{sect: BC for K} below slightly more general boundary conditions for the case $m+1 \leq j \leq i$.
They are of the form
$$
k_{ij}(x,0) =\varphi_{ij}(x) + \sum_{\alpha=1}^n k_{i\alpha}(x,0) b_{\alpha j},
$$
for some scalars $b_{\alpha j} \in \R$.
Then, we can check that the kernel equations are still well-posed, provided that the matrix $\Id-B$ is invertible, where $B$ is the transpose of the matrix $(b_{\alpha j})_{m+1 \leq \alpha,j \leq i}$.
This condition gives that the above boundary condition can be rewritten into a boundary condition of a similar form but without the indices $\alpha=m+1, \ldots, i$.
This allows to use similar arguments to those in the references mentioned above to prove the well-posedness.
\end{remark}

\subsection{Minimal control time for full row rank boundary coupling matrices}
\label{sect rankQ=p}

We will now recall two important situations where the minimal control time for the \NC of systems of the form \eqref{syst G} is known.
The first one is when the boundary coupling matrix $Q$ is full row rank, i.e. $\rank Q=p$.
Note that, in that case, $\NC$ and $\EC$ are equivalent properties, just as they are for system \eqref{syst}.
We have the following result:

\begin{theorem}\label{thm HO21-JMPA syst G}
Assume that $\rank Q=p$.
Then, the minimal control time for the \EC of
system \eqref{syst G} does not depend on $G$ and is thus equal to
$$
\max \ens{
T_m, \quad \max_{1 \leq k \leq p} \, T_{m+k}+T_{c_k}
}.
$$
\end{theorem}

This result can be proved with the method developed in \cite{HO21-JMPA}.
Alternatively, we can also use the general bounds obtained in \cite{HO22-JDE} for such systems (see Theorem 1.11 therein and its proof on pages 678-679 and 695) and notice that they coincide when $\rank Q=p$.

\subsection{Minimal control time for $1+p$ systems}
\label{sect 1+p syst}

From the work \cite{HO21-COCV}, another situation where the minimal control time of system \eqref{syst G} is known is the case $m=1$.
In that case, $Q \in \R^p$ is just a vector and $r_1$ simply denotes the position of the first nonzero component of $Q$ (when it exists, i.e. $Q \neq 0$).
To recall this result, we need to introduce the following quantity:

\begin{definition}\label{def xf}
For $x_0 \in (0,1]$ and a function $f:(0,x_0) \longrightarrow \R$, we denote by $\Xm{x_0}{f} \in [0,x_0]$ the length of the largest interval of the form $(0,\ell)$ with $\ell \leq x_0$ where the function $f$ vanishes a.e.
\end{definition}

We can now recall \cite[Theorem 5.1]{HO21-COCV}:

\begin{theorem}\label{thm 1+p}
Assume that $m=1$ and denote by $g_1,\ldots,g_p$ the components of $G_{+-}$.
Then, the minimal control time for the \NC of
system \eqref{syst G} is
$$
\max\ens{T_2, \quad \max_{\substack{1 \leq k \leq p \\ k<r_1}}  T_{1+k}^{I_{1+k}}+T_1, \quad T_{1+r_1}+T_1},
$$
where $I_{1+k}$ is the interval $(\Xm{1}{g_k},1)$.
\end{theorem}

\subsection{Removal of part of the internal coupling}

Let us now recall several general simplifications obtained in \cite{HO22-JDE} regarding the coupling structure of system \eqref{syst G}.
To this end, we introduce the block decomposition
$$
G=
\begin{pmatrix}
G_{--} \\
G_{+-}
\end{pmatrix},
\quad \begin{array}{l}
G_{--} \in \R^{m \times m}, \\ G_{+-} \in \R^{p \times m}.
\end{array}
$$

Then, \cite[Theorem 7.1]{HO22-JDE} states

\begin{theorem}
The minimal control time for the \NC of system \eqref{syst G} does not depend on $G_{--}$.
\end{theorem}

Consequently, we only have to focus on systems of the form \eqref{syst G} with $G_{--}=0$.
In the sequel, such systems will be denoted by
$$(Q,G_{+-}).$$

\subsection{Reduction to canonical boundary coupling matrices}

It has also been shown in \cite{HO22-JDE} that $Q$ can be changed into its canonical form, up to changing $G_{+-}$ as well.
More precisely, we have

\begin{proposition}\label{prop LCU decomp}
Let $(L,U)$ be a pair giving $LQU=\Qcan$.
Let $T>0$.
System $(Q,G_{+-})$ is \NC in time $T$ if, and only if, so is system $(\Qcan,\Gcan_{+-})$, where $\Gcan_{+-}(x)$ is the matrix whose $j$-th row is
\begin{equation}\label{def Gcan}
\Gcan_{m+j}(x)=\sum_{k \leq j} l_{jk}G_{m+k}(\zeta_{m+j,m+k}(x))U,
\end{equation}
where $\zeta_{m+j,m+k}$ is the solution $\zeta$ to the ODE
$$
\begin{dcases}
\zeta'(s)=\frac{\lambda_{m+k}(\zeta(s))}{\lambda_{m+j}(s)}, \quad s \in \R, \\
\zeta(0)=0.
\end{dcases}
$$
\end{proposition}

This result is contained in the proof of \cite[Proposition 4.1]{HO22-JDE}.

\section{Design of boundary conditions for the kernel equations}
\label{sect: BC for K}

Using a second Volterra integral transformation, it has further been shown in \cite{HO22-JDE} that we can change any system $(\Qcan,G_{+-})$, with $\Qcan$ in canonical form and $G_{+-}$ arbitrary, into an equivalent system $(\Qcan,\tilde{G}_{+-})$, with $\tilde{G}_{+-}$ having the structure
$$
\tilde{g}_{m+j,c_\gamma}=0 \quad \text{ whenever } j \geq r_\gamma.
$$
This structure simply means that $\tilde{G}_{+-}$ has zeroes at and below its $(r_\gamma,c_\gamma)$ entries.
For instance, if the canonical form of $Q$ is
$$
\Qcan=
\begin{pmatrix}
1 & 0 \\
0 & 0 \\
0 & 1
\end{pmatrix},
$$
then this means that $\tilde{G}_{+-}$ has the form
$$
\tilde{G}_{+-}=
\begin{pmatrix}
0 & \times \\
0 & \times \\
0 & 0
\end{pmatrix}.
$$
We refer to \cite[Section 6]{HO22-JDE} for details.

Despite sufficient for our purpose (the proof of Theorem \ref{main thm}), we show here that we can actually find a special kernel $K$ such that $\Gcan_{+-}$ defined by \eqref{def Gcan}-\eqref{def G} already inherits such a structure.
This thus makes the use of an additional Volterra integral transformation unnecessary in our situation.
This result is new and important in itself.

\begin{theorem}\label{new thm kern}
There exists a solution $K$ to the kernel equations \eqref{kern equ} such that $\Gcan_{+-}$ defined by \eqref{def Gcan} and \eqref{def G} satisfies
\begin{equation}\label{cond Gtilde}
\gcan_{m+j,c_\gamma}=0 \quad \text{ whenever } j \geq r_\gamma.
\end{equation}
\end{theorem}

\begin{proof}
Recall that the rows of a solution $K$ to the kernel equations are decoupled.
We can see \eqref{cond Gtilde} as a condition on $\Gcan_{m+j}$, the $j$-th row of $\Gcan_{+-}$.
Since $\Gcan_{m+j}$ depends only on the $(m+j)$-th row of $K$, we do not need to discuss about the first $m$ rows of $K$ (in other words, we can take any solution for those rows).
We are now going to show how to define appropriate boundary conditions for the $(m+j)$-th row of $K$, by induction on $j$ and starting with $j=1$.
\begin{enumerate}
\item
First of all, if $r_1>1$, then for the $(m+j)$-th rows of $K$ with $j<r_1$, we fix arbitrary boundary conditions (whenever needed and appropriately, at $x=0$ or $\xi=1$, see Remark \ref{rmk wp kern}).
Then, starting from the $r_1$-th row onward, we proceed as follows.

\item
Let us introduce $G_{+-}^U=G_{+-}U$.
We first write what condition \eqref{cond Gtilde} means for $G_{+-}^U$.
By definition of $\Gcan_{+-}$ (see \eqref{def Gcan}), we obviously have
$$
\Gcan_{m+j}(x)=\sum_{k \leq j} l_{jk}G_{m+k}^U(\zeta_{m+j,m+k}(x)).
$$
Separating the term for $k=j$ from the rest and owing to the facts that $l_{jj}=1$ and $\zeta_{m+j,m+j}(x)=x$, condition \eqref{cond Gtilde} can thus be written as
\begin{equation}\label{cond Gtilde 2}
g_{m+j,c_\gamma}^U(x)=\varphi_{m+j,c_\gamma}(x) \quad \text{ whenever } j \geq r_\gamma,
\end{equation}
where we introduced $\varphi_{m+j,c_\gamma}(x)= -\sum_{k<j} l_{jk}g_{m+k,c_\gamma}^U(\zeta_{m+j,m+k}(x))$.
It is crucial to observe that $\varphi_{m+j,c_\gamma}$ depends only on rows of $G_{+-}^U$ before $j$, and thus only on rows of $\begin{pmatrix} K_{+-} & K_{++} \end{pmatrix}$ before $j$.
This is what will make the induction work.

\item
By definition of $G_{+-}$ and since $LQU=\Qcan$, we have
$$
G_{+-}^U(x)=-K_{+-}(x,0)\Lambda_-(0)U
-K_{++}(x,0)\Lambda_+(0)L^{-1} \Qcan.
$$
Let us write it as
$$G_{+-}^U(x)=F(x) +H(x) \Qcan,$$
where
$$F(x)=-K_{+-}(x,0)\Lambda_-(0)U, \quad H(x)=-K_{++}(x,0)\Lambda_+(0)L^{-1}.
$$
To follow our previous notations, the $(\alpha,\beta)$ entry of $F$ (resp. $H$) will be denoted by $f_{m+\alpha,\beta}$ (resp. $h_{m+\alpha,m+\beta}$).
Then, condition \eqref{cond Gtilde 2} becomes a condition on $H$ as follows:
\begin{equation}\label{cond Gtilde 3}
h_{m+j,m+r_\gamma}(x)=
-f_{m+j,c_\gamma}(x)
+\varphi_{m+j,c_\gamma}(x) \quad \text{ whenever } j \geq r_\gamma.
\end{equation}

\item
Let us now write condition \eqref{cond Gtilde 3} as a condition on $K_{++}$.
According to Theorem \ref{thm kern}, the values of $K_{++}(x,0)$ that can be chosen are exactly $k_{m+j,m+\beta}(x,0)$ with $j \geq \beta$.
By definition of $H$ and since $l^{r_\gamma r_\gamma}=1$, we have
$$
h_{m+j,m+r_\gamma}(x)=
-k_{m+j,m+r_\gamma}(x,0) \lambda_{m+r_\gamma}(0)
-\sum_{\beta>r_\gamma} k_{m+j,m+\beta}(x,0) \lambda_{m+\beta}(0) l^{\beta r_\gamma}.
$$
Let now $j \geq r_1$ be arbitrary.
Let $N \geq 1$ be the largest integer such that $r_N \leq j$.
First, for each $\beta \notin \ens{r_1,\ldots,r_N}$, we fix an arbitrary boundary condition for $k_{m+j,m+\beta}$ (at $\xi=0$ or $x=1$).
Then, we solve
\begin{multline*}
k_{m+j,m+r_\gamma}(x,0)
=
\frac{1}{\lambda_{m+r_\gamma}(0)} \Bigg(
-\sum_{\beta>r_\gamma} k_{m+j,m+\beta}(x,0) \lambda_{m+\beta}(0) l^{\beta r_\gamma}
\\
-(-f_{m+j,c_\gamma}(x)
+\varphi_{m+j,c_\gamma}(x))
\Bigg),
\end{multline*}
for all $\gamma \in \ens{1,\ldots,N}$.
This is possible since, on the left-hand side we have $j \geq r_\gamma$ and these boundary conditions are of the form stated in Remark \ref{rmk wp kern}, with corresponding matrix $B$ strictly upper triangular.

\end{enumerate}

\end{proof}

\section{Proof of the main result}
\label{sect proof main thm}

We have now all the ingredients necessary to prove our main result.

\begin{proof}[Proof of Theorem \ref{main thm}]
\begin{enumerate}
\item
By assumption of the theorem, the canonical form of $Q$ is
\begin{equation}\label{def Qcan}
\Qcan=
\left(\begin{array}{c|c}
J & 0 \\
\hline
0 & e
\end{array}\right),
\end{equation}
where:
\begin{itemize}
\item
$J \in \R^{(m-1)\times(m-1)}$ is invertible,

\item
$e \in \R^{p-(m-1)}$ is equal to the $r_m$-th canonical vector if $\rank Q=m$ and to zero otherwise.
\end{itemize}

We fix a kernel $K$ as in Theorem \ref{new thm kern}.
Then, the corresponding $\Gcan_{+-}$ has the form
\begin{equation}\label{def Ghat}
\Gcan_{+-}=
\left(\begin{array}{c|c}
A & b \\
\hline
0 & d
\end{array}\right),
\end{equation}
for some $A \in \R^{(m-1) \times (m-1)}$, $b \in \R^{m-1}$ and $d \in \R^{p-(m-1)}$.
From the discussion in Section \ref{sect prelim}, we have to show that the minimal time for the \NC of system $(\Qcan, \Gcan_{+-})$ is equal to the quantity claimed in the statement of Theorem \ref{main thm}.
For convenience, let us write this system explicitly:
\begin{equation}\label{syst final}
\begin{dcases}
\pt{y_-}(t,x)+\Lambda_-(x) \px{y_-}(t,x)=0, \\
\pt{y_+}(t,x)+\Lambda_+(x) \px{y_+}(t,x)=\Gcan_{+-}(x) y_-(t,0), \\
y_-(t,1)=u(t), \quad y_+(t,0)=\Qcan y_-(t,0), \\
\IC,
\end{dcases}
\quad (t,x) \in R_T,
\end{equation}
where $\Lambda_-=\diag(\lambda_1,\ldots,\lambda_m)$ and $\Lambda_+=\diag(\lambda_{m+1},\ldots,\lambda_{m+p})$.

\item
Because of the structures \eqref{def Qcan} and \eqref{def Ghat} of $\Qcan$ and $\Gcan_{+-}$, we split system \eqref{syst final} into two parts by separating the first $m-1$ components of $y_-$ and $y_+$ from the remaining ones.
Thus, we write
$$
y_-=
\left(\begin{array}{c}
y_1 \\
\vdots \\
y_{m-1} \\
\hline
y_m
\end{array}\right)
=
\begin{pmatrix}
\yIIm \\
\yIm
\end{pmatrix},
\quad
y_+=
\left(\begin{array}{c}
y_{m+1} \\
\vdots \\
y_{2m-1} \\
\hline
y_{2m} \\
\vdots \\
y_{m+p}
\end{array}\right)
=
\begin{pmatrix}
\yIIp \\
\yIp
\end{pmatrix},
$$
with $\yIIm \in \R^{m-1}$, $\yIm \in \R$ and $\yIIp \in \R^{m-1}$, $\yIp \in \R^{p-(m-1)}$.
Similarly, we use the decomposition
$$
u=
\left(\begin{array}{c}
u_1 \\
\vdots \\
u_{m-1} \\
\hline
u_m
\end{array}\right)
=
\begin{pmatrix}
\uII \\
\uI
\end{pmatrix},
$$
with $\uII \in \R^{m-1}$ and $\uI \in \R$.
Finally, we also need to introduce
$$
\Lambda_-=\diag(\LambdaIIm,\LambdaIm), \quad \Lambda_+=\diag(\LambdaIIp,\LambdaIp),
$$
with $\LambdaIIm \in \R^{(m-1)\times(m-1)}$, $\LambdaIm \in \R$ and $\LambdaIIp \in \R^{(m-1)\times(m-1)}$, $\LambdaIp \in \R^{(p-(m-1))\times(p-(m-1))}$.
Then, system \eqref{syst final} takes the form
\begin{equation}\label{syst final decomp}
\begin{dcases}
\pt{\yIIm}(t,x)+\LambdaIIm(x) \px{\yIIm}(t,x)=0, \\
\pt{\yIm}(t,x)+\LambdaIm(x) \px{\yIm}(t,x)=0, \\
\pt{\yIIp}(t,x)+\LambdaIIp(x) \px{\yIIp}(t,x)=A(x) \yIIm(t,0) + b(x) \yIm(t,0), \\
\pt{\yIp}(t,x)+\LambdaIp(x) \px{\yIp}(t,x)=d(x) \yIm(t,0), \\
\yIIm(t,1)=\uII(t), \quad \yIIp(t,0)=J\yIIm(t,0), \\
\yIm(t,1)=\uI(t), \quad \yIp(t,0)=e\yIm(t,0), \\
\IC,
\end{dcases}
\quad (t,x) \in R_T.
\end{equation}
We see that this system can be divided into two subsystems.
However, the order is important.

\item
First, we control ``subsystem I'':
\begin{equation}\label{subsyst I}
\begin{dcases}
\pt{\yIm}(t,x)+\LambdaIm(x) \px{\yIm}(t,x)=0, \\
\pt{\yIp}(t,x)+\LambdaIp(x) \px{\yIp}(t,x)=d(x) \yIm(t,0), \\
\yIm(t,1)=\uI(t), \quad \yIp(t,0)=e\yIm(t,0), \\
\IC,
\end{dcases}
\quad (t,x) \in R_T.
\end{equation}
By Theorem \ref{thm 1+p}, we know that the minimal control time of this ``$1+\times$'' subsystem is
$$
\TI=\max\ens{T_{2m}, \quad \max_{\substack{m \leq k \leq p \\ k<r_m}}  T_{m+k}^{I_{m+k}}+T_m, \quad T_{m+r_m}+T_m},
$$
with $I_{m+k}=(\Xm{1}{d_k},1)$.
Once a control $\uI$ is fixed so that $\yIm(T,\cdot)=\yIp(T,\cdot)=0$, the states $\yIm$ and $\yIp$ become known quantities.

\item
Denoting by $f(t,x)=b(x) \yIm(t,0)$, we now control ``subsystem II'':
$$
\begin{dcases}
\pt{\yIIm}(t,x)+\LambdaIIm(x) \px{\yIIm}(t,x)=0, \\
\pt{\yIIp}(t,x)+\LambdaIIp(x) \px{\yIIp}(t,x)=A(x) \yIIm(t,0) + f(t,x), \\
\yIIm(t,1)=\uII(t), \quad \yIIp(t,0)=J\yIIm(t,0), \\
\IC,
\end{dcases}
\quad (t,x) \in R_T.
$$
This is a system with invertible boundary coupling matrix.
Therefore, from Theorem \ref{thm HO21-JMPA syst G}, this system with $f=0$ is \EC in any time $T>\TII$, where
$$\TII=\max_{1 \leq k \leq m-1} T_{m+k}+T_{c_k},$$
(obviously, the indices $c_k$ of $J$ and $Q$ are the same).
Since we are dealing with \EC, the same remains true for arbitrary $f$.
In summary, we have shown that system \eqref{syst final decomp} is \NC in any time $T>\max\ens{\TI,\TII}$.
Clearly, $\max\ens{\TI,\TII}=\max\ens{\tau^1, \tau^2}$, with $\tau^1, \tau^2$ given as in the statement of the theorem.

\item
Let us now explain why system \eqref{syst final decomp} cannot be \NC before $\max\ens{\TI,\TII}$.

\begin{itemize}
\item
First, it is clear $T \geq \TI$ is necessary since, if system \eqref{syst final decomp} is \NC in time $T$, then so is its subsystem \eqref{subsyst I} and, as mentioned above, $\TI$ is the minimal control time for this subsystem.

\item
To see why $T \geq \TII$ is also necessary we recall that, as mentioned in the introduction (see also \cite[Theorem 5.1 (i)]{HO22-JDE}), the minimal control time cannot be smaller than the minimal control time of the same system without internal coupling, which is exactly $\max\ens{\TII, T_{m+r_m}+T_m}$ for system \eqref{syst final decomp}.
\end{itemize}
The proof of Theorem \ref{main thm} is complete.

\end{enumerate}
\end{proof}

\begin{remark}
Let us mention a particular structure, for the initial system \eqref{syst}, where the proof above is simplified (yet, the main ideas are the same).
It is when $Q$ satisfies the assumptions of the theorem, $\rank Q=m-1$, and
$$
M=
\begin{pmatrix}
\times & \times \\
0 & \times
\end{pmatrix},
$$
where the zero block is of size $p \times m$ (that is, there is no internal coupling term for $y_-$ in the equations of $y_+$).
In that case, we have $\Tinf=\max\ens{\tau^1, T_m}$ and the proof is simplified as follows.
First, the structure of $M$ implies that we can find a solution $K$ to the kernel equations with the same structure.
Together with the assumptions on $Q$, this implies that the last column of $\Gcan_{+-}$ is zero.
Thus, in the proof above, subsystems I and II and decoupled.
Moreover, since $e=d=0$, we do not even have to invoke Theorem \ref{thm 1+p} to see that the minimal control time for subsystem I is just $\max\ens{T_{2m}, T_m}$.
\end{remark}

\section{Example}
\label{sect example}

In this last part, we use the system of Example \ref{ex} to show how the intervals $I_{m+k}$ of our main result can sometimes be explicitly characterized in function of the parameters $\Lambda,M,Q$ of the system.
Let us first recall the framework.
We consider system \eqref{syst} with parameters
$$
Q=
\begin{pmatrix}
1 & 0 \\
0 & 0 
\end{pmatrix},
\quad
M=
\left(\begin{array}{cc|cc}
0 & 0 & \times & \times \\
\times & 0 & \times & \times \\
\hline
\times & 0 & 0 & \times \\
\times & f & \times & 0 
\end{array}\right),
$$
where $f \in L^{\infty}(0,1)$ is arbitrary.
We are thus in the case $m=p=2$, the matrix $Q$ is already in canonical form and, clearly, $(r_1,c_1)=(1,1)$.
From Theorem \ref{main thm}, the minimal time for the \NC of the corresponding system is equal to
$$
\Tinf=\max\ens{T_3+T_1, \, T_4^{I_4}+T_2},
$$
where, from the proof of the theorem, $I_4$ is the interval $(\Xm{1}{\gcan_{42}},1)$.

We now characterize this interval in function of the parameters $\Lambda,M,Q$ only.
Since $Q$ is already in canonical form, we have $\Gcan_{+-}=G_{+-}$ (see Proposition \ref{prop LCU decomp}).
By definition of $G$ (see \eqref{def G}) and since the second column of $Q$ is zero, we have
$$g_{42}(x)= -k_{42}(x,0) \lambda_2(0).$$
Using that all but the last entry of the second column of $M$ are also zeros, we see that $k_{42}$ satisfies the following equation (see \eqref{kern equ}, note that $M^0=M$ since the diagonal of $M$ is already full of zeros):
$$
\begin{dcases}
\lambda_4(x)\px{k_{42}}(x,\xi)
+\pxi{k_{42}}(x,\xi)\lambda_2(\xi)
+k_{42}(x,\xi) \lambda_2'(\xi)
+k_{44}(x,\xi) f(\xi)
=0,
\\
k_{42}(x,x) =\frac{f(x)}{\lambda_4(x) -\lambda_2(x)},
\end{dcases}
\quad (x,\xi) \in \Tau.
$$
It remains to link intervals of the form $(0,\times)$ where $k_{42}(\cdot,0)$ and $f$ vanish.
This was done in \cite[Section 4]{HO21-COCV} and the following identity was derived:
$$
T_2+T_4^{(\Xm{1}{k_{42}(\cdot,0)}, 1)}=
T_2^{(\Xm{x_0}{f}, 1)}
+T_4^{(\Xm{x_0}{f}, 1)},
$$
where $x_0 \in (0,1)$ is the unique solution to $T_4^{(0,x_0)}+T_2^{(0,x_0)}=T_4$.

\section*{Acknowledgments}

The second author would like to thank Shandong University for its hospitality.
This work was initiated while he was visiting there.
This project was supported by National Natural Science Foundation of China (Nos. 12526214, 12521001, 12471421) and National Science Centre, Poland UMO-2023/50/E/ST1/00081.

\renewcommand*{\bibfont}{\small}
\printbibliography


\end{document}